\documentclass[a4paper]{amsart}
\usepackage{graphicx}
\theoremstyle{plain}
  \newtheorem{thm}{Theorem}[section]
  \newtheorem{lem}[thm]{Lemma}
  \newtheorem{cor}[thm]{Corollary}
  \newtheorem{prop}[thm]{Proposition}

  \newtheorem*{cornt}{Corollary~\ref{cor:nottrivial}}
  \newtheorem*{propinf}{Proposition~\ref{prop:infinite}}
\theoremstyle{definition}
  
  \newtheorem{ex}[thm]{Example}

\theoremstyle{remark}
  \newtheorem{rem}[thm]{Remark}
  \newtheorem*{ack}{Acknowledgments}
\newcommand{\Z}{\mathbb{Z}}
\newcommand{\cyclic}[1]{\Z/#1\Z}
\newcommand{\cyclicproduct}[2]{(\cyclic{#1})^{#2}}

\newcommand{\cover}[1]{e(#1)}
\newcommand{\lift}[1]{l(#1)}

\newcommand{\trivial}{\emptyset}

\newcommand{\cvector}[1]{\vec{#1}}
\newcommand{\svector}[2]{\vec{#1}_{#2}}

\newcommand{\gpset}[1]{\mathcal{GP}(#1)}
\newcommand{\gpmset}[1]{\mathcal{GPM}(#1)}

\DeclareMathOperator{\height}{height}
\DeclareMathOperator{\base}{base}

\numberwithin{equation}{section}
\allowdisplaybreaks
\begin{document}
\title{Homotopy invariants of Gauss words}
\author{Andrew Gibson}
\address{
Department of Mathematics,
Tokyo Institute of Technology,
Oh-okayama, Meguro, Tokyo 152-8551, Japan
}
\email{gibson@math.titech.ac.jp}
\date{\today}
\begin{abstract}
By defining combinatorial moves, we can define an equivalence relation
 on Gauss words called homotopy.
In this paper we define a homotopy invariant of Gauss words.
We use this to show that there exist Gauss words that are not
 homotopically equivalent to the empty Gauss word, disproving a
 conjecture by Turaev.
In fact, we show that there are an infinite number of equivalence
 classes of Gauss words under homotopy. 
\end{abstract}
\thanks{This work was supported by a Scholarship from the Ministry of
Education, Culture, Sports, Science and Technology of Japan.} 
\maketitle
\section{Introduction}
A Gauss word is a sequence of letters with the condition that any letter
appearing in the sequence does so exactly twice. 
A Gauss word can be obtained from an oriented virtual knot
diagram. Given a diagram, label the real crossings and arbitrarily pick
a base point on the curve somewhere away from any of the real
crossings. Starting from the base point, we follow the curve and read off
the labels of the crossings as we pass through them. When we return to
the base point we will have a sequence of letters in which each label of
a real crossing appears exactly twice. Thus this sequence is a Gauss
word.
\par
If we give the real crossings different labels we will get a different
Gauss word from the diagram.
We wish to consider all such Gauss words equivalent, so we introduce the
idea of an isomorphism Gauss words.
Two Gauss words are isomorphic if there is a bijection between the sets
of letters in the Gauss words, which transforms one Gauss word into the
other.
Diagrammatically, an isomorphism corresponds to relabelling the real
crossings.
\par
Depending on where we introduce the base point, we may get different
Gauss words from a single diagram.
To remove this dependence we introduce a combinatorial move on a Gauss
word.
The move allows us to remove the initial letter from the word and
append it to the end of the word.
We call this move a shift move.
Diagrammatically, the shift move corresponds to the base point being
moved along the curve through a single real crossing.
We can then say that, modulo the shift move and isomorphism, the
representation of an oriented virtual knot diagram by a Gauss word is
unique. 
\par
We can define an equivalence relation on virtual knot diagrams by
defining diagrammatic moves called generalized Reidemeister moves. Two
diagrams are defined to be equivalent if there exists a finite sequence
of such moves transforming one diagram into the other. Virtual knots are
defined to be the equivalence classes of this relation.
\par
By analogy, we can define combinatorial moves on Gauss words which
correspond to the generalized Reidemeister moves. We can then use these
moves, the shift move and isomorphism to define an equivalence relation
on Gauss words which we call homotopy.
We define the moves in such a way that if two virtual knot diagrams
represent the same virtual knot then the Gauss words obtained from the
diagrams must be equivalent under our combinatorial moves. This means
that if we have two virtual knot diagrams for which the associated Gauss
words are not equivalent, we can immediately say that the diagrams
represent different virtual knots.
\par
If we disallow the shift move, we can define another kind of homotopy of
Gauss words which we call open homotopy.
It is clear from the definition that if two Gauss words are open
homotopic, they must also be homotopic.
In this paper we will show that the opposite conclusion is not
necessarily true.
In other words, we will show that homotopy and open homotopy of Gauss
words are different.
\par
In \cite{Turaev:Words}, Turaev introduced the idea of
nanowords which are defined as Gauss words with some associated data.
By introducing moves on nanowords, different kinds of homotopy can be
defined. 
From this viewpoint, homotopy of Gauss words is the simplest kind of
nanoword homotopy.
In fact, any Gauss word homotopy invariant is an invariant for any
kind of homotopy of nanowords.
\par
In \cite{Turaev:Words}, Turaev defined several invariants of nanowords.
However, all these invariants are trivial in the case of Gauss words.
This led Turaev to conjecture that open homotopy of Gauss words is
trivial.  
That is, he conjectured that every Gauss word is open homotopically
equivalent to the empty Gauss word.
\par
In this paper we define a homotopy invariant of Gauss words called $z$
which takes values in an abelian group.
This invariant was inspired by Henrich's smoothing invariant for virtual
knots defined in \cite{Henrich:vknots}.
In fact, our invariant can be viewed as a version of Henrich's
invariant, weakened sufficiently to remain invariant under homotopy of
Gauss words.
\par
We give an example of a Gauss word for which $z$ takes a different value
to that of the empty Gauss word.
This shows that Turaev's conjecture is false.
We state this result as follows.
\begin{cornt}
There exist Gauss words that are not homotopically trivial.
\end{cornt}
Using the idea of a covering of a Gauss word, which was originally
introduced by Turaev in \cite{Turaev:Words}, we define the height of a
Gauss word.
This is a homotopy invariant and we use it to prove the following
proposition. 
\begin{propinf}
There are an infinite number of homotopy classes of Gauss words.
\end{propinf}
\par
We also give an invariant for open homotopy of Gauss words called
$z_o$ which is defined in a similar way to $z$.
The invariant $z$ itself can be viewed as an open homotopy invariant.
However, we show that $z_o$ is a stronger invariant than $z$.
\par
The rest of this paper is arranged as follows.
In Section~\ref{sec:gausswords} we give a formal definition of Gauss
words, homotopy and open homotopy. 
In Section~\ref{sec:gaussphrases} we describe Gauss phrases and recall the
$S$ invariant which was defined in \cite{Gibson:gauss-phrase}.
\par
In Section~\ref{sec:invariant} we give the definition of $z$ and prove
its invariance under homotopy. 
We then use $z$ to give an example of a Gauss word that is not
homotopically trivial. 
In Section~\ref{sec:coverings} we recall the definition of the covering
invariant from \cite{Turaev:Words}. We use this invariant to
show how we can construct infinite families of Gauss words which are
mutually non-homotopic.
\par
In Section~\ref{sec:openhomotopy} we describe the open homotopy
invariant $z_o$.
We use this invariant to show that open homotopy is different from
homotopy. 
\par
In Section~\ref{sec:virtualknots} we interpret the existence of Gauss
words that are not homotopically trivial in terms of moves on virtual
knot diagrams.
\par
Having written this paper, we discovered a paper by Manturov
\cite{Manturov:freeknots} which studies objects called free knots.
From Lemma 1 in Manturov's paper and the discussion in
Section~\ref{sec:virtualknots} it is clear that an oriented free knot is
equivalent to a homotopy class of Gauss words.
Manturov shows the existence of non-trivial free knots which implies the
result we give in Corollary~\ref{cor:nottrivial}.
Our Propostion~\ref{prop:infinite} can be deduced from his results.
In Manturov's paper objects corresponding to open homotopy classes of
Gauss words are not considered.
However, it is clear that his invariant can be generalized to this case.
\par
We also found a second paper on free knots written by Manturov
\cite{Manturov:freeknotslinks}.
In this paper he defines an invariant which is essentially the same as
our $z$ invariant.
\begin{ack}
The author would like to thank his supervisor, Hitoshi Murakami, for all
 his help and advice.
\end{ack}
\section{Gauss words}\label{sec:gausswords}
An \emph{alphabet} is a finite set and its elements are
called \emph{letters}.
A \emph{word} on an alphabet $\mathcal{A}$ is a map $m$ from an ordered
finite set $\{1,\dotsc,n\}$ to $\mathcal{A}$. 
Here $n$ is called the \emph{length} of the word.
It is a non-negative integer.
We usually write a word as sequence of letters $m(1)m(2)\dotso m(n)$
from which the map can be deduced if needed.
For example, $ABBB$ and $CABAACA$ are both words on the alphabet
$\{A,B,C\}$.
For any alphabet there is a unique empty word of length $0$ which we
write $\trivial$.
\par
A \emph{Gauss word} on an alphabet $\mathcal{A}$ is a word on
$\mathcal{A}$ such that every letter in $\mathcal{A}$ appears in the
word exactly twice.
We define the \emph{rank} of a Gauss word to be the size of
$\mathcal{A}$.
This means that the rank of a Gauss word is always half its length.
\par
For example, $ABAB$ and $ABBA$ are both Gauss words on the alphabet
$\{A,B\}$. They both have length $4$ and rank $2$.
\par
There is a unique trivial Gauss word $\trivial$ on an empty alphabet
which has length and rank $0$.
\par
By definition, the alphabet that a Gauss word is defined on is the set
of letters appearing in the Gauss word.
Therefore, we do not need to explicitly state the alphabet that the
Gauss word is defined on.
\par
Let $u$ be a Gauss word on $\mathcal{A}$ and $v$ be a Gauss word on
$\mathcal{B}$.
An \emph{isomorphism} of Gauss words $u$ and $v$ is a bijection $f$ from
$\mathcal{A}$ to $\mathcal{B}$ such that $f$ applied letterwise to $u$
gives $v$. 
If such an isomorphism exists we say that $u$ and $v$ are isomorphic.
\par
We now define some combinatorial moves on Gauss words.
If we have a Gauss word matching the pattern on the left of the move we
may transform it to the pattern on the right or vice-versa.
In each move $t$, $x$, $y$ and $z$ represent possibly empty, arbitrary
sequences of letters and $A$, $B$ and $C$ represent individual letters.
The moves are
\begin{equation*}
\begin{array}{ll}
\textrm{Shift:}\quad & AxAy \longleftrightarrow xAyA, \\
\textrm{H1:}\quad & xAAy \longleftrightarrow xy, \\
\textrm{H2:}\quad & xAByBAz \longleftrightarrow xyz, \\
\textrm{H3:}\quad & xAByACzBCt \longleftrightarrow xBAyCAzCBt \\
\end{array}
\end{equation*}
and are collectively known as \emph{homotopy moves} and were originally
defined by Turaev in \cite{Turaev:Words}.
\par
Two Gauss words are \emph{homotopic} if there exists a finite sequence
of isomorphisms and homotopy moves which transforms one into the other.
This relation is an equivalence relation which we call \emph{homotopy}.
It divides the set of Gauss words into \emph{homotopy classes}.
We define the \emph{homotopy rank} of a Gauss word $w$ to be the minimum
rank of all the Gauss words that are homotopic to $w$.
We say that a Gauss word is \emph{homotopically trivial} if it is
homotopic to the trivial Gauss word $\trivial$.
Such a Gauss word has homotopy rank $0$.
\par
If we disallow the shift move, we get a potentially different kind of
homotopy which we call \emph{open homotopy}. 
It is easy to see that if two Gauss words are open homotopic, they must
be homotopic.
We will show later that the reverse is not necessarily true.
\par
Homotopy of Gauss words is the simplest kind of homotopy of nanowords.
Turaev defined nanowords in \cite{Turaev:Words}.
A nanoword is a Gauss word with a map, called a \emph{projection}, from
its alphabet to some set $\alpha$.
An isomorphism of nanowords is an isomorphism of Gauss words which
preserves this projection.
A particular homotopy of nanowords is determined by fixing $\alpha$ and
specifying some other information known collectively as \emph{homotopy
data} (see \cite{Turaev:Words} for full details). 
Moves on nanowords are defined in the same way as moves on Gauss words.
However, restrictions dependent on the projection, homotopy data and
$\alpha$ limit when the moves can be applied.
\par
Homotopy on nanowords is defined analogously to homotopy of Gauss
words.
That is, two nanowords are homotopic if there exists a finite sequence
of isomorphisms and homotopy moves transforming one nanoword into the
other. 
In \cite{Turaev:Words}, Turaev defines homotopy of nanowords
without allowing the shift move.
In this paper we call this kind of homotopy \emph{open homotopy} of
nanowords. 
\par
In this general setting, homotopy of Gauss words is a homotopy of
nanowords where the set $\alpha$ is a single element.  
\par
In \cite{Turaev:Words}, Turaev derived some other moves from the
homotopy moves H1, H2 and H3 for nanowords. These hold for Gauss words
and are
\begin{equation*}
\begin{array}{ll}
\textrm{H2a:}\quad & xAByABz \longleftrightarrow xyz \\
\textrm{H3a:}\quad & xAByCAzBCt \longleftrightarrow xBAyACzCBt \\
\textrm{H3b:}\quad & xAByCAzCBt \longleftrightarrow xBAyACzBCt \\
\textrm{H3c:}\quad & xAByACzCBt \longleftrightarrow xBAyCAzBCt. \\
\end{array}
\end{equation*}
\section{Gauss phrases}\label{sec:gaussphrases}
A \emph{phrase} is a finite sequence of words $w_1,\dotsc,w_n$ on some
alphabet. 
We call each word in the sequence a \emph{component} of the phrase.
If the concatenation $w_1\dotso w_n$ of all words in a phrase gives a
Gauss word, we say that the phrase is a \emph{Gauss phrase}.
A Gauss phrase with only one component is necessarily a Gauss word.
\par
In this paper we write Gauss phrases as a sequence of letters, using a
$|$ to separate components.
So, for example, $ABA|B$ is a Gauss phrase written in this way.
\par
Let $p$ and $q$ be Gauss phrases with $n$ components.
We write $p$ as $u_1|\dotso|u_n$ and $q$ as $v_1|\dotso|v_n$.
Then $p$ and $q$ are isomorphic if there exists a bijection $f$ from
the alphabet of $p$ to the alphabet of $q$ such that $f$ applied
letterwise to $u_i$ gives $v_i$ for all $i$. 
\par
We define the homotopy moves H1, H2 and H3 for Gauss phrases in the same
way as we did for Gauss words.
We modify the meaning of the letters $t$, $x$, $y$ and $z$ in these
moves to allow for the inclusion of one or more occurences of the
component separator $|$.
Note that a move cannot be applied if the component separator $|$
appears between the letters in the subwords $AA$, $AB$, $BA$, $AB$, $AC$
and $BC$ that are explicitly shown in the moves.
For example, given the Gauss phrase $AB|BAC|C$, we may apply the move H2
to get the Gauss phrase $\trivial|C|C$, but we cannot apply the move H1
to remove the letter $C$.
\par
We define a shift move for Gauss phrases which can be applied to a
single component of the Gauss phrase. 
Suppose $p$ is a Gauss phrase with $i$th component of the form $Ax$ for
some letter $A$ and some letter sequence $x$.
The shift move applied to the $i$th component of $p$ gives a Gauss
phrase $q$ where the $i$th component has the form $xA$ and every other
component of $q$ matches that of $p$.
\par
We say two Gauss phrases are \emph{homotopic} if there is a finite sequence of
isomorphisms, shift moves and the moves H1, H2 and H3
 which transform one Gauss phrase into the other.
This relation is an equivalence relation on Gauss phrases called
\emph{homotopy}.
\par
None of the moves on Gauss phrases allows a component to be added or
removed. 
Thus, the number of components of a Gauss phrase is a homotopy
invariant.
As Gauss words are one component Gauss phrases, we can see that homotopy
of Gauss phrases is a generalization of the homotopy of Gauss words.
\par
As we did for Gauss words, we can define \emph{open homotopy} of Gauss
phrases by disallowing the shift move.
In fact we can define various kinds of homotopy on $n$-component Gauss
phrases by only allowing shift moves on a subset of the components.
A component for which the shift move is permitted is called
\emph{closed} and one for which the shift move is not permitted is
called \emph{open}.
Thus under homotopy of Gauss phrases all components are closed and under
open homotopy all components are open.
In this paper we use the term \emph{mixed homotopy} to mean the homotopy
on $2$-component Gauss phrases where the first component is closed and
the second one is open. 
\par
By allowing permutations of components of a Gauss phrase we can define
another kind of homotopy.
In this paper we only consider this kind of homotopy when all the
components are closed.
We call this homotopy \emph{unordered homotopy}.
\par
We studied homotopy of Gauss phrases in \cite{Gibson:gauss-phrase}. 
In that paper we defined a homotopy invariant of Gauss phrases called
the $S$ invariant.
We recall the definition here.
\par
Let $p$ be an $n$-component Gauss phrase.
We write $K_n$ for $\cyclicproduct{2}{n}$.
\par
Given a vector $\cvector{v}$ in $K_n$ we can define a map
$c_{\cvector{v}}$ from $K_n$ to itself as follows
\begin{equation*}
c_{\cvector{v}}(\cvector{x}) = \cvector{v} - \cvector{x} \mod 2.
\end{equation*}
In \cite{Gibson:gauss-phrase} we showed that $c_{\cvector{v}}$ is either
the identity map or an involution.
This means that the orbits of $K_n$ under $c_{\cvector{v}}$ all contain at
most two elements.
We define $K(\cvector{v})$ to be the set of orbits of $K_n$ under
$c_{\cvector{v}}$.
\par
For any subword $u$ of a single component in $p$ we define the
\emph{linking vector} of $u$ to be a vector $\cvector{v}$ in $K_n$.
The $i$th element of 
$\cvector{v}$ is defined to be, modulo 2, the number of letters that
appear once in $u$ and for which the other occurence appears in the
$i$th component of $p$.
\par
Let $w_k$ be the $k$th component of $p$. As $w_k$ can be considered a
subword of the $k$th component, we define the linking vector of the
$k$th component to be the linking vector of $w_k$.
We write this vector $\svector{l}{k}$.
\par
For any letter $A$ that appears twice in the same component of $p$, that
component must have the form $xAyAz$ for some, possibly empty, arbitrary
sequences of letters $x$, $y$ and $z$.
We define the linking vector of $A$ to be the linking vector of the
subword $y$.
\par
Write $[\cvector{0}]$ for the orbit of $\cvector{0}$ in
$K(\svector{l}{k})$.
For the $k$th component in $p$ we define a subset $O_k(p)$ of 
$K(\svector{l}{k})-\{[\cvector{0}]\}$ as follows.
Let $A_k$ be the set of letters which appear twice in the $k$th
component of $p$.
Then an orbit $v$ in $K(\svector{l}{k})-\{[\cvector{0}]\}$ is in $O_k(p)$
if there are an odd number of letters in $A_k$ for which the
linking vector of the letter is in $v$.
\par
We define $S_k(p)$ to be the pair $(\svector{l}{k}, O_k(p))$.
We then define $S$ to be the $n$-tuple where the $k$th element of $S$ is
$S_k(p)$.
In \cite{Gibson:gauss-phrase} we showed that $S$ is a homotopy
invariant of $p$.
\par
We can represent $S_k(p)$ as a matrix.
To do this, we first define an order on $K_n$ as follows.
Let $\cvector{u}$ and $\cvector{v}$ be vectors in $K_n$.
Let $j$ be the smallest integer for which the $j$th elements of
$\cvector{u}$ and $\cvector{v}$ differ.
Then $\cvector{u}$ is smaller than $\cvector{v}$ if the $j$th element of
$\cvector{u}$ is $0$.
\par
We define the \emph{representative vector} of an orbit $v$ in
$K(\svector{l}{k})$ to be the smallest vector in that orbit.
Let $R$ be the set of representative vectors of orbits which are in
$O_k(p)$.
Let $r$ be the number of vectors in $R$.
\par
We construct a matrix with $n$ columns and $r+1$ rows for $S_k(p)$.
The first row is the vector $\svector{l}{k}$.
The remaining $r$ rows are given by the elements of $R$ written out in
ascending order.
In \cite{Gibson:gauss-phrase} we observed that this construction is
canonical in the following sense.
Given $n$-component Gauss phrases $p$ and $q$, $S_k(p)$ and $S_k(q)$ are
equivalent if and only if their matrix representations are equal. 
Thus we can write $S$ as an $n$-tuple of matrices.
\par
As preparation for Section~\ref{sec:invariant}, we consider what happens
to the invariant $S$ under unordered homotopy of $2$-component Gauss
phrases.
\par 
Let $p$ be a $2$-component Gauss phrase.
Then $p$ has the form $w_1|w_2$.
Define $q$ to be the $2$-component Gauss phrase $w_2|w_1$.
Under unordered homotopy, $p$ and $q$ are equivalent.
We compare $S(p)$ and $S(q)$.
\par
We can write $S(p)$ as a pair of matrices $(M_1,M_2)$ where both $M_1$
and $M_2$ have $2$ columns.
Given a two column matrix $M$ we define a new matrix $T(M)$ by
\begin{equation*}
T(M) = M 
\begin{pmatrix}
0 & 1 \\
1 & 0 \\
\end{pmatrix}.
\end{equation*}
The matrix $T(M)$ is the same size as $M$ but has its columns in the
opposite order.
Using this notation it is easy to check that we can write $S(q)$ as
$(T(M_2),T(M_1))$.
We say that $S(p)$ and $S(q)$ are related by a \emph{transposition}.
If we consider the invariant $S$ modulo transposition, we get an
unordered homotopy invariant of $2$-component Gauss phrases.
\begin{rem}
It is possible to define a similar equivalence relation on the $S$
 invariant to obtain an unordered homotopy invariant in the general
 $n$-component case.
However, as we do not need to use such an invariant in this paper, we do
 not give a definition here.
\end{rem}
\par
In \cite{Gibson:gauss-phrase} we also defined an open homotopy invariant
of Gauss phrases similar to $S$ which we called $S_o$.
We can construct a hybrid of the two invariants, called $S_m$, which is
a mixed homotopy invariant of $2$-component Gauss phrases.
Recall that under mixed homotopy the first component is closed and the
second component is open.
We now give a definition of $S_m$.
\par
Given a $2$-component Gauss phrase $p$, we define $S_m$ to be the pair
of pairs $((\svector{l}{1},O_1(p)),(\svector{l}{2},B_2(p)))$ where
$\svector{l}{1}$, $\svector{l}{2}$ and $O_1(p)$ are 
defined as for $S$.
We define $B_2(p)$ to be a subset of $K_n-\{\cvector{0}\}$ as follows.
Let $A_2$ be the set of letters which appear twice in the second
component.
Then a vector $\cvector{v}$ in $K_n-\{\cvector{0}\}$
is in $B_2(p)$ if there are an odd number of 
letters in $A_2$ for which the
linking vector of the letter is $\cvector{v}$. 
\par
With reference to \cite{Gibson:gauss-phrase} it is easy to check that
$S_m$ is a mixed homotopy invariant. 
We will use $S_m$ in Section~\ref{sec:openhomotopy}.
\section{Homotopy invariant}\label{sec:invariant}
Let $\gpset{2}$ be the set of equivalence classes of $2$-component Gauss
phrases under unordered homotopy.
Let $G$ be the free abelian group generated by $\gpset{2}$.
We then define $Z$ to be $G/2G$.
Let $\theta$ be the natural homomorphism from $G$ to $Z$.
\par
Let $w$ be a Gauss word. 
For each letter $A$ in $w$ we can derive a $2$-component Gauss phrase
$p(w,A)$ as follows.
As $A$ must appear twice in $w$, $w$ has the form $xAyAz$ for some,
possibly empty, sequences of letters $x$, $y$ and $z$.
Then $p(w,A)$ is the Gauss phrase $y|xz$.
We then define $u(w,A)$ to be the unordered homotopy class of $p(w,A)$,
an element in $\gpset{2}$.
We define $t(w)$ to be the element in $\gpset{2}$ given by the unordered
homotopy class of $\trivial|w$. 
\par
We define a map $g$ from the set of Gauss words to $G$ as follows.
For each Gauss word $w$, $g(w)$ is given by
\begin{equation*}
g(w) = \sum_{A \in w} \left( u(w,A) - t(w) \right).
\end{equation*}
Then $z(w)$, defined to be $\theta(g(w))$, gives a map $z$
from the set of Gauss words to $Z$.
\par
We have the following theorem.
\begin{thm}\label{thm:invariant_z}
The map $z$ is a homotopy invariant of Gauss words.
\end{thm}
\begin{proof}
We need to prove that if two Gauss words $w_1$ and $w_2$ are homotopic,
 $z(w_1)$ is equal to $z(w_2)$.
In order to do this, it is sufficient to prove that $z$ is invariant
 under isomorphism, the shift move and the moves H1, H2 and H3.
We note that if $w_1$ and $w_2$ are homotopic, $t(w_1)$ is equal to
 $t(w_2)$ by definition.
\par
Suppose $w_1$ and $w_2$ are isomorphic Gauss words.
Then each letter $A$ in $w_1$ is mapped to some letter $A^\prime$ in
 $w_2$ by some isomorphism $f$. 
So $p(w_2,A^\prime)$ is isomorphic to $p(w_1,A)$ under the isomorphism
 $f$ restricted to all the letters in $w_1$ except for $A$.
This means that $u(w_2,A^\prime)$ is equal to $u(w_1,A)$.
As this is the case for every letter in $w_1$ we can conclude that
 $g(w_1)$ is equal to $g(w_2)$.
In particular, this means that $z(w_1)$ is equal to $z(w_2)$ and so $z$
 is invariant under isomorphism.
\par
Suppose $w_1$ and $w_2$ are related by a shift move.
Then $w_1$ is of the form $Av$ and $w_2$ is of the form $vA$.
\par
Let $B$ be some other letter appearing in $w_1$.
Then $w_1$ has the form $AxByBz$ and $w_2$ has the form $xByBzA$.
Now $p(w_1,B)$ is $y|Axz$ and $p(w_2,B)$ is $y|xzA$.
Applying a shift move to the second component of $p(w_1,B)$ we get
 $p(w_2,B)$. 
This means $u(w_1,B)$ is equal to $u(w_2,B)$.
\par
We now turn our attention to the letter $A$.
We can write $w_1$ in the form $AxAy$ and $w_2$ in the form $xAyA$.
Then $p(w_1,A)$ is $x|y$ and $p(w_2,A)$ is $y|x$.
As Gauss phrases, $x|y$ and $y|x$ are not necessarily equal.
However, they are related by a permutation and so are equivalent under
 unordered homotopy.
Thus $u(w_1,A)$ is equal to $u(w_2,A)$.
\par
As $u(w_1,X)$ equals $u(w_2,X)$ for each letter $X$ in $w_1$, we can
 conclude that  $g(w_1)$ is equal to $g(w_2)$.
Thus $z$ is invariant under the shift move.
\par
Suppose $w_1$ and $w_2$ are related by an H1 move.
Then $w_1$ has the form $xAAz$ and $w_2$ has the form $xz$.
Now for any letter $B$ in $w_1$ other than $A$, $p(w_1,B)$ will contain
 the subword $AA$ in one or other of its two components.
Thus this subword can be removed from by $p(w_1,B)$ an H1 move.
The result is the Gauss phrase $p(w_2,B)$ which implies that $u(w_1,B)$
 equals $u(w_2,B)$.
Therefore, if we subtract $g(w_2)$ from $g(w_1)$ we get
$u(w_1,A) - t(w_1)$. 
Now as $p(w_1,A)$ is $\trivial|xz$ which is homotopic to $\trivial|w_1$,
 $u(w_1,A)$ is equal to $t(w_1)$.
Thus $g(w_2)$ is equal to $g(w_1)$ and so $z$ is invariant under the H1
 move.
\par
Suppose $w_1$ and $w_2$ are related by an H2 move.
Then $w_1$ has the form $xAByBAz$ and $w_2$ has the form $xyz$.
It is easy to see that for any letter $C$ in $w_1$,
 other than $A$ or $B$, $p(w_1,C)$ and $p(w_2,C)$ will be related by an
 H2 move involving $A$ and $B$.
Thus $u(w_1,C)$ equals $u(w_2,C)$ for all such letters $C$.
Now $p(w_1,A)$ is $ByB|xz$.
By applying a shift move to the first component and then an H1 move, we
 can remove the letter $B$ and get a Gauss phrase $y|xz$.
On the other hand, $p(w_1,B)$ is $y|xAAz$.
Applying an H1 move to the second component, we get the Gauss phrase
 $y|xz$.
Thus $p(w_1,A)$ is homotopic to $p(w_1,B)$ and so $u(w_1,A)$ equals
 $u(w_1,B)$.
We can conclude that if we subtract $g(w_2)$ from $g(w_1)$ we get
 $2u(w_1,A) - 2t(w_1)$.
As $2u(w_1,A) - 2t(w_1)$ is in the kernel of $\theta$, $z(w_1)$ equals
 $z(w_2)$ and $z$ is invariant under the H2 move.
\par
Finally, suppose $w_1$ and $w_2$ are related by an H3 move.
Then $w_1$ has the form $tABxACyBCz$ and $w_2$ has the form
 $tBAxCAyCBz$.
\par
Let $D$ be some letter in $w_1$ other than $A$, $B$ or $C$.
There are $10$ possible cases depending on where the two occurences of
 $D$ occur in relation to the subwords $AB$, $AC$ and $BC$.
In Table~\ref{tab:h3invariance} we show that in all $10$ cases,
 $p(w_1,D)$ is homotopic to $p(w_2,D)$ by the homotopy move shown in the
 column furthest to the right. 
Thus in every case, $u(w_1,D)$ is equal to $u(w_2,D)$.
\par
\begin{table}[hbt]
\begin{center}
\begin{tabular}{r|c|c|c|c}
Case & $w_1$ & $p(w_1,D)$ & $p(w_2,D)$ & Move \\
\hline
 1 & $rDsDtABxACyBCz$ & $s|rtABxACyBCz$ & $s|rtBAxCAyCBz$ & H3 \\
 2 & $rDsABtDxACyBCz$ & $sABt|rxACyBCz$ & $sBAt|rxCAyCBz$ & H3 \\
 3 & $rDsABtACxDyBCz$ & $sABtACx|ryBCz$ & $sBAtCAx|ryCBz$ & H3 \\
 4 & $rDsABtACxBCyDz$ & $sABtACxBCy|rz$ & $sBAtCAxCBy|rz$ & H3 \\
 5 & $rABsDtDxACyBCz$ & $t|rABsxACyBCz$ & $t|rBAsxCAyCBz$ & H3 \\
 6 & $rABsDtACxDyBCz$ & $tACx|rABsyBCz$ & $tCAx|rBAsyCBz$ & H3c \\
 7 & $rABsDtACxBCyDz$ & $tACxBCy|rABsz$ & $tCAxCBy|rBAsz$ & H3c \\
 8 & $rABsACtDxDyBCz$ & $x|rABsACtyBCz$ & $x|rBAsCAtyCBz$ & H3 \\
 9 & $rABsACtDxBCyDz$ & $xBCy|rABsACtz$ & $xCBy|rBAsCAtz$ & H3b \\
10 & $rABsACtBCxDyDz$ & $y|rABsACtBCxz$ & $y|rBAsCAtCBxz$ & H3 \\
\end{tabular}
\end{center}
\caption{Invariance under homotopy of a Gauss phrase associated with a
 letter uninvolved in an H3 move}
\label{tab:h3invariance}
\end{table}
We now turn our attention to the letters $A$, $B$ and $C$.
\par
In the case of $A$, $p(w_1,A)$ is $Bx|tCyBCz$ and $p(w_2,A)$ is
 $xC|tByCBz$. 
Applying a shift move to the first component of $p(w_1,A)$ we get
 $xB|tCyBCz$ which is isomorphic to $p(w_2,A)$.
Thus $u(w_1,A)$ is equal to $u(w_2,A)$.
\par
In the case of $B$, $p(w_1,B)$ is $xACy|tACz$.
By applying an H2a move to remove the letters $A$ and $C$ we get the
 Gauss phrase $xy|tz$.
On the other hand, $p(w_2,B)$ is $AxCAyC|tz$.
We can apply a shift move to the first component followed by an H2a
 move to remove the letters $A$ and $C$.
We again get the Gauss phrase $xy|tz$.
Thus $p(w_1,B)$ and $p(w_2,B)$ are homotopic and so $u(w_1,B)$ is equal
 to $u(w_2,B)$.
\par
In the case of $C$, $p(w_1,C)$ is $yB|tABxAz$ and $p(w_2,C)$ is
 $Ay|tBAxBz$. 
Applying a shift move to the first component of $p(w_2,C)$ we get the Gauss
 phrase $yA|tBAxBz$ which is isomorphic to $p(w_1,C)$.
Thus $u(w_1,C)$ is equal to $u(w_2,C)$.
\par
So we have seen that $u(w_1,X)$ is equal to $u(w_2,X)$ for all letters
 $X$ in $w_1$. 
Therefore $g(w_1)$ is equal to $g(w_2)$ and we can conclude that $z$ is
 invariant under the H3 move.
\end{proof}
We now calculate this invariant for two examples.
\begin{ex}
Consider the trivial Gauss word $\trivial$.
As there are no letters in $\trivial$, $g(\trivial)$ is $0$ and so
 $z(\trivial)$ is $0$. 
\end{ex}
\begin{ex}\label{ex:counter}
Let $w$ be the Gauss word $ABACDCEBED$.
We calculate $z(w)$.
\par
The Gauss phrase $p(w,A)$ is $B|CDCEBED$. By using a shift move on the
 second component and applying an H2a move, we can remove $C$ and $D$.
After applying another shift move to the second component we can use an H1
 to remove $E$.
Thus $p(w,A)$ is homotopic to $B|B$.
\par
The Gauss phrase $p(w,B)$ is $ACDCE|AED$.
Applying two shift moves to the first component gives $DCEAC|AED$.
We can then remove $A$ and $E$ by an H2 move and then, using an H1 move,
 remove $C$.
This means that $p(w,B)$ is homotopic to $D|D$.
\par
The Gauss phrase $p(w,C)$ is $D|ABAEBED$.
We can apply an H3 move to the letters $A$, $B$ and $E$.
The result is $D|BAEAEBD$. 
The letters $A$ and $E$ can be removed by an H2a move and then the
 letter $B$ can be removed by an H1 move.
This shows that $p(w,C)$ is homotopic to $D|D$.
\par
The Gauss phrase $p(w,D)$ is $CEBE|ABAC$.
\par
The Gauss phrase $p(w,E)$ is $B|ABACDCD$.
Applying a shift move to the second component gives $B|BACDCDA$.
We can then use an H2a move and an H1 move to remove the letters $C$,
 $D$ and $A$.
This shows that $p(w,E)$ is homotopic to $B|B$.
\par
From these calculations we can see that $p(w,A)$, $p(w,B)$, $p(w,C)$ and 
 $p(w,E)$ are all mutually homotopic.
Thus we can write $g(w)$ as $4u(w,A) + u(w,D) - 5t(w)$.
Now observe that $\theta(4u(w,A))$ is in the kernel of $\theta$ and that
 $\theta(- 5t(w))$ is equal to $\theta(t(w))$.
So $z(w)$ is equal to $\theta (u(w,D)) + \theta(t(w))$.
\par
We calculate the $S$ invariant for $p(w,D)$.
We find that it is the pair of matrices
\begin{equation}\label{eq:s-ud}
\begin{pmatrix}
\begin{pmatrix}
0 & 0 \\
0 & 1 \\
\end{pmatrix},
\begin{pmatrix}
0 & 0 \\
1 & 0 \\
\end{pmatrix}
\end{pmatrix}.
\end{equation}
The transposition of $S(p(w,D))$ gives the same pair of matrices.
This is unsurprising because swapping the components in $CEBE|ABAC$
 gives $ABAC|CEBE$ which, after applying three shift moves to the first
 component and one to the second, is isomorphic to $CEBE|ABAC$.
\par
Calculating $S(\trivial|w)$ and $S(w|\trivial)$, we find they are both
 given by 
\begin{equation}\label{eq:s-trivial}
\begin{pmatrix}
\begin{pmatrix}
0 & 0 \\
\end{pmatrix},
\begin{pmatrix}
0 & 0 \\
\end{pmatrix}
\end{pmatrix}.
\end{equation}
As the pairs of matrices in \eqref{eq:s-ud} and \eqref{eq:s-trivial} are
 not equal, we can conclude that $u(w,D)$ is not equal to $t(w)$.
Thus $z(w)$ is not equal to $0$.
As we have already shown that $z(\trivial)$ is $0$, we can conclude that
 $w$ is not homotopically trivial.
\end{ex}
\par
From these examples we can make the following conclusion.
\begin{cor}\label{cor:nottrivial}
There exist Gauss words that are not homotopically trivial.
\end{cor}
\par
If two nanowords are homotopic then their associated Gauss words must be
homotopic.
This means that the $z$ invariant is an invariant of any homotopy of
nanowords.
As virtual knots can be represented as homotopy classes of a certain
homotopy of nanowords \cite{Turaev:KnotsAndWords}, this shows that there
exist non-trivial virtual knots for which non-triviality can be
determined by just considering their associated Gauss word homotopy
class.
\par
\begin{figure}[hbt]
\begin{center}
\includegraphics{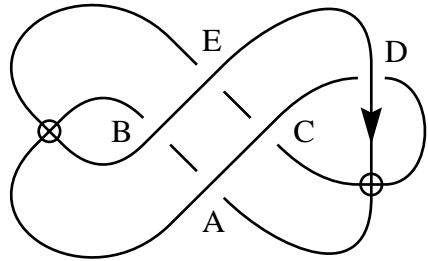}
\caption{A non-trivial virtual knot}
\label{fig:5knot}
\end{center}
\end{figure}
For example, consider the virtual knot shown in Figure~\ref{fig:5knot}.
In the figure, crossings with small circles are virtual crossings (see
Section~\ref{sec:virtualknots} for further details about virtual knots
and their diagrams). 
By taking the arrow as the base point, we can derive a Gauss word from
the diagram.
We find it is the Gauss word considered in Example~\ref{ex:counter}.
From that example we know that the Gauss word is not homotopically
trivial.
Thus we can conclude that the virtual knot is non-trivial.
\par
When we calculate the $z$ invariant for a nanoword, we do so without
reference to its associated projection.
However, using information in the projection it should be possible to
define stronger invariants based on $z$.
Indeed, as we mentioned in the introduction, Henrich's smoothing
invariant \cite{Henrich:vknots} is an example of such an invariant in the
case of virtual knots.
\section{Coverings}\label{sec:coverings}
In Section~5 of \cite{Turaev:Words} Turaev introduced an open
homotopy invariant of nanowords called a covering.
This invariant is a map from the set of open homotopy classes of
nanowords to itself.
In this section we give an alternative but equivalent definition of this
invariant for Gauss words and show that it is a homotopy invariant.
We will then use this invariant to contruct infinite families of
mutually non-homotopic Gauss words.
\par
Given a Gauss word $w$ and a letter $A$ appearing in $w$, $w$ has the
form $xAyAz$ for some, possibly empty, arbitrary sequences of letters
$x$, $y$ and $z$.
We say that $A$ has \emph{odd parity} if $y$ has an odd number of
letters and $A$ has \emph{even parity} if $y$ has an even number of
letters.
We define the covering of $w$ to be a copy of $w$
where all the odd parity letters have been removed. We denote the
covering of $w$ by $\cover{w}$.
In Turaev's notation $\cover{w}$ would be denoted $w^H$ where $H$ is the
trivial group (see Section~5.2 of \cite{Turaev:Words} for further
details).
\begin{ex}\label{ex:cover}
Let $w$ be the Gauss word $ABCADBECED$. Then the letters $B$ and $E$
 have odd parity and the other letters are even parity. So
 $\cover{w}$ is the Gauss word $ACADCD$.
\end{ex}
\begin{prop}
The homotopy class of the covering of a Gauss word $w$ is a homotopy
 invariant of $w$.
\end{prop}
\begin{proof}
As Turaev has already proved this fact for open homotopy in Lemma~5.2.1
of \cite{Turaev:Words}, it is sufficient to prove invariance
under the shift move.
\par
Given a Gauss word $u$ of the form $AxAy$, the shift move transforms it
 to a Gauss word of the form $xAyA$ which we label $v$.
Since the length of $u$ is even, the lengths of $x$ and $y$ have the
 same parity.
This means that the parity of $A$ is the same in $u$ and $v$.
For any letter other than $A$, it is clear that the parity of the letter
 is the same in $u$ and $v$.
Thus the parity of any letter in $u$ is invariant under the shift move.
\par
Suppose $A$ has even parity.
Then $\cover{u}$ is given by $Ax^\prime Ay^\prime$ for some words
 $x^\prime$ and $y^\prime$ derived from $x$ and $y$ by deleting odd
 parity letters.  
The covering of $v$, $\cover{v}$, is then $x^\prime Ay^\prime A$.
So $\cover{v}$ can be obtained from $\cover{u}$ by a shift move
 which means they are homotopic.
\par
Suppose $A$ has odd parity.
Then $\cover{u}$ is given by $x^\prime y^\prime$ for some words
 $x^\prime$ and $y^\prime$ and $\cover{v}$ is also given by
 $x^\prime y^\prime$. 
Thus $\cover{u}$ and $\cover{v}$ are equal and therefore homotopic.
\end{proof}
\begin{rem}
In fact, we can define this kind of covering for any homotopy of
nanowords. This is because the parity of a letter in a Gauss word can be
calculated without any reference to the projection of the nanoword.
In this general setting we call this covering the \emph{even parity
 covering}. 
We first gave a definition of the even parity covering of nanowords in
\cite{Gibson:mthesis}.
\end{rem}
\par
We have the following lemma.
\begin{lem}\label{lem:reduce}
Let $w$ be a Gauss word.
If $w$ is not homotopic to $\cover{w}$ then the homotopy rank of
 $\cover{w}$ is strictly less than the homotopy rank of $w$. 
\end{lem}
\begin{proof}
Let $m$ be the homotopy rank of $w$.
Then we can find a Gauss word $w^\prime$ which is homotopic to $w$ and
 has rank $m$.
Consider $\cover{w^\prime}$.
If the rank of $\cover{w^\prime}$ is $m$ it means that
 $\cover{w^\prime}$ is the same as $w^\prime$.
However, as $\cover{w^\prime}$ is homotopic to $\cover{w}$, this
 implies $\cover{w}$ is homotopic to $w$, contradicting the
 assumption of the lemma.
Thus the rank of $\cover{w^\prime}$ cannot be $m$ and
 $\cover{w^\prime}$ is different from $w^\prime$.
As, by definition, we derive $\cover{w^\prime}$ from $w^\prime$ by
 removing letters, we must conclude that the rank of
 $\cover{w^\prime}$ is less than $m$.
Thus, the homotopy rank of $\cover{w}$, which is less than or equal
 to the rank of $\cover{w^\prime}$, is less than $m$.
\end{proof}
\begin{rem}
In fact, we can say that the rank of $\cover{w^\prime}$ must be less
 than $m-1$.
For if the rank of $\cover{w^\prime}$ was $m-1$ it would
 mean that we obtained $\cover{w^\prime}$ by removing a single letter
 from $w^\prime$.
This in turn would imply that $w^\prime$ had a single odd parity letter.
However, this would contradict Lemma~5.2 of \cite{Gibson:gauss-phrase},
 which states that any Gauss word has an even number of odd parity
 letters.
\end{rem}
As the covering of a Gauss word is itself a Gauss word, we can
repeatedly take coverings to form an infinite sequence of Gauss words.
That is, given a Gauss word $w$, define $w_0$ to be $w$ and define $w_i$
to be $\cover{w_{i-1}}$ for all positive integers $i$.
As $w$ has a finite number of letters, Lemma~\ref{lem:reduce} shows that
there must exist an $n$ for which $w_{n+1}$ is homotopic to $w_n$.
Let $m$ be the smallest such $n$.
We define the height of $w$, $\height(w)$, to be $m$ and the base of
$w$, $\base(w)$ to be the homotopy class of $w_m$.
The height and base of $w$ are homotopy invariants of $w$.
\par
In \cite{Gibson:ccc} we defined height and base invariants for virtual
strings in the same way.
We showed that the base invariants are non-trivial for virtual strings
in that paper.
However, we do not know whether the base invariant we have defined here
is non-trivial for Gauss words.
In other words, we have not yet found a Gauss word $w$ for which
we can prove $\base(w)$ is not homotopically trivial.
\par
Given a Gauss word $w$, we can define a new Gauss word $v$ such that
$\cover{v}$ is $w$.
We start by taking a copy of $w$. Then for each odd parity letter $A$ in
$w$ we replace the first occurence of $A$ with $XAX$ for some letter $X$
not already appearing in $w$.
Note that this replacement changes the parity of $A$ to make it
even. 
The parity of any other letter in the word is unchanged because we
replace a subword of length $1$ with a subword of length $3$.
Note also that the introduced letter $X$ has odd parity.
After making the change for each odd parity letter in $w$,
we call the final Gauss word $\lift{w}$.
\par
By construction all the letters in $\lift{w}$ that were originally in
$w$ have even parity and all the letters that were introduced have odd
parity. 
Thus, when we take the covering of $\lift{w}$ we remove all the letters
that we introduced and we are left with the letters in $w$. Since we did
not change the order of the letters, we can conclude that
$\cover{\lift{w}}$ is equal to $w$.
\begin{ex}
Let $w$ be the Gauss word $ABCADBECED$ from Example~\ref{ex:cover}.
There are only two odd parity letters in $w$, $B$ and $E$.
We replace the first occurence of $B$ with $XBX$ and the first occurence
 of $E$ with $YEY$.
The result is the Gauss word $AXBXCADBYEYCED$ which we label
 $\lift{w}$.
The covering of $\lift{w}$ is $w$.
\end{ex}
We remark that if a Gauss word $w$ contains no odd parity letters, then
$\lift{w}$ is the same as $w$.
Even if $w$ and $\lift{w}$ are not equal as Gauss words, they may be
homotopic.
We provide an example to show this. 
\begin{ex}
Consider the Gauss word $w$ given by $ABAB$.
Then $\lift{w}$ is given by $XAXYBYAB$. 
By move H3c on $X$, $A$ and $Y$, $\lift{w}$ is homotopic to
$AXYXBAYB$. Applying a shift move we get $XYXBAYBA$ which can be reduced
to $XYXY$ by an H2a move involving $A$ and $B$. This Gauss word is
isomorphic to $w$. Thus $\lift{w}$ and $w$ are homotopic. 
\end{ex}
\par
Given a Gauss word $w$ we can make an infinite family of Gauss words
$w_i$ by repeated use of this construction.
We define $w_0$ to be $w$.
Then inductively we define $w_i$ to be $\lift{w_{i-1}}$ for all positive
integers $i$.
\begin{lem}\label{lem:infinite-family}
Let $w$ be a Gauss word such that $\cover{w}$ and $w$ are not
 homotopic.
Let $w_i$ be the infinite family of Gauss words defined from $w$ as
 above.
Then the $w_i$ are all mutually non-homotopic. 
\end{lem}
\begin{proof}
Suppose, for some $i$, $w_{i+1}$ is homotopic to $w_i$.
By construction, starting from $w_{i+1}$ and taking the covering $i+1$
 times, we get $w$.
Similarly, starting from $w_i$, taking the covering $i+1$ times gets
 $\cover{w}$.
Since the covering of a Gauss word is a homotopy invariant, our
 supposition implies that $w$ is homotopic to $\cover{w}$.
However this contradicts the assumption in the statement of the lemma.
Therefore, for all $i$, $w_{i+1}$ is not homotopic to $w_i$.
\par
As $\cover{w_{i+1}}$ is $w_i$, this implies $\height(w_{i+1})$ is equal
 to $\height(w_i)+1$ for all $i$.
It is now simple to prove that $\height(w_i)$ is $\height(w)+i$ by
 induction.
\par
As each Gauss word $w_i$ has a different height we can conclude that
 they are all mutually non-homotopic.
\end{proof}
\begin{prop}\label{prop:infinite}
There are an infinite number of homotopy classes of Gauss words.
\end{prop}
\begin{proof}
By Lemma~\ref{lem:infinite-family} we just need to give an example of a
 Gauss word which is not homotopic to its cover.
Consider the Gauss word $w$ given by $ABACDCEBED$.
Then $\cover{w}$ is $DD$ which is homotopic to the trivial Gauss
 word.
On the other hand, in Example~\ref{ex:counter}
we saw that $w$ is not homotopic to the trivial Gauss word.
Thus $w$ and $\cover{w}$ are not homotopic.
\end{proof}
\section{Open homotopy}\label{sec:openhomotopy}
Although the invariant $z$ is an invariant for open homotopy of Gauss
words, we can use a similar construction to make a stronger invariant
for open homotopy. We call this invariant $z_o$.
\par
Let $\gpmset{2}$ be the set of equivalence classes of $2$-component Gauss
phrases under mixed homotopy.
Recall that we defined this to be the homotopy where the first component
is closed and the second is open.
Let $H$ be the free abelian group generated by $\gpmset{2}$.
We then define $Z_o$ to be $H/2H$.
Let $\phi$ be the natural homomorphism from $H$ to $Z_o$.
\par
For a Gauss word $w$ and a letter $A$ appearing in $w$, we define
$u_m(w,A)$ to be the equivalence class in $\gpmset{2}$ which contains
$p(w,A)$.
Here $p(w,A)$ is the Gauss phrase defined in
Section~\ref{sec:invariant}.
We define $t_m(w)$ to be the element in $\gpmset{2}$ which contains
$\trivial|w$.
\par
For each Gauss word $w$, we define $h(w)$ by
\begin{equation*}
h(w) = \sum_{A \in w} \left( u_m(w,A) - t_m(w) \right).
\end{equation*}
Then $h$ is a map from the set of Gauss words to $H$.
We then define $z_o(w)$ to be $\phi(h(w))$.
Thus $z_o$ is a map from the set of Gauss words to $Z_o$.
\par
We have the following theorem.
\begin{thm}
The map $z_o$ is an open homotopy invariant of Gauss words.
\end{thm}
\begin{proof}
The fundamental difference between the definitions of $z_o$ and $z$ is
 the type of homotopy we use to determine equivalence of Gauss phrases.
For $z$ we consider elements in $\gpset{2}$ whereas for $z_o$ we
 consider elements in $\gpmset{2}$.
Note that we can consider $\gpset{2}$ to be $\gpmset{2}$ modulo
 permutation of the two components and allowing shift moves on the
 second component. 
\par
Looking at the proof of the invariance of $z$
 (Theorem~\ref{thm:invariant_z}) we note the following two facts.
Firstly, the permutation of components is only used for the proof of
 invariance under the shift move.
Secondly, we only need to apply a shift move to the second component of
 a Gauss phrase in the proof of invariance under the shift move.
\par
Thus, by changing the notation appropriately and omitting the section
 about invariance under the shift move, the proof of
 Theorem~\ref{thm:invariant_z} becomes a proof of the invariance of
 $z_o$ under open homotopy.  
\end{proof}
We now give an example of a Gauss phrase which is trivial under homotopy
but non-trivial under open homotopy.
This shows that homotopy and open homotopy of Gauss phrases are
different.
\begin{ex}\label{ex:oh}
Let $w$ be the Gauss word $ABACDCBD$.
By an H3c move applied to $A$, $B$ and $C$, $w$ is homotopic to
 $BACADBCD$.
Applying a shift move we get $ACADBCDB$ which is homotopic to $ACAC$ by
 an H2a move. 
Applying another H2a move gives the empty Gauss word.
Thus $ABACDCBD$ is trivial under homotopy.
\par
We now calculate $z_o(w)$ to show that $w$ is not trivial under open
 homotopy.
We start by calculating the Gauss phrases for each letter.
We find that $p(w,A)$ is $B|CDCBD$, $p(w,B)$ is $ACDC|AD$, $p(w,C)$ is
 $D|ABABD$ and $p(w,D)$ is $CB|ABAC$.
Using an H2a move, we see that $p(w,C)$ is homotopic to $D|D$.
\par
We then calculate $S_m$ for each of the four Gauss phrases.
We find that
\begin{align*}
S_m(B|CDCBD)
&=
\begin{pmatrix}
\begin{pmatrix}
0 & 1 \\
\end{pmatrix},
\begin{pmatrix}
1 & 0 \\
0 & 1 \\
1 & 1 \\
\end{pmatrix}
\end{pmatrix}, \\
S_m(ACDC|AD)
&= 
\begin{pmatrix}
\begin{pmatrix}
0 & 0 \\
0 & 1 \\
\end{pmatrix},
\begin{pmatrix}
0 & 0 \\
\end{pmatrix}
\end{pmatrix}, \\
S_m(D|D)
&=
\begin{pmatrix}
\begin{pmatrix}
0 & 1 \\
\end{pmatrix},
\begin{pmatrix}
1 & 0 \\
\end{pmatrix}
\end{pmatrix} \\
\intertext{and}
S_m(CB|ABAC)
&= 
\begin{pmatrix}
\begin{pmatrix}
0 & 0 \\
\end{pmatrix},
\begin{pmatrix}
0 & 0 \\
1 & 0 \\
\end{pmatrix}
\end{pmatrix}.
\end{align*}
Therefore they are mutually distinct under mixed homotopy.
\par
Noting that $4t_m(w)$ is in the kernel of $\phi$, we have
\begin{equation*}
z_o(w) = 
\phi(\langle B|CDCBD \rangle) +
\phi(\langle ACDC|AD \rangle) +
\phi(\langle D|D \rangle) +
\phi(\langle CB|ABAC \rangle)
\ne 0,
\end{equation*}
where $\langle q \rangle$ represents the mixed homotopy equivalence
 class of a Gauss phrase $q$.
As $z_o(w)$ is not zero, we can conclude that $w$ is not trivial under
 open homotopy.
\end{ex}
\par
Note that we can define the height and base of
Gauss words under open homotopy in the same way as we did for Gauss
words under homotopy in Section~\ref{sec:coverings}.
We write $\height_o(w)$ for the height of $w$ and $\base_o(w)$ for the
base of $w$ under open homotopy.
For a given Gauss word $w$, $\height(w)$ and $\height_o(w)$ are not
necessarily the same.
For example, if $w$ is the Gauss word $ABACDCBD$ in Example~\ref{ex:oh},
$\height(w)$ is $0$ but $\height_o(w)$ is $1$.
We do not know whether there exists a Gauss word $w$ for which
$\base(w)$ and $\base_o(w)$ are different. 
\par
We end this section by remarking that there are an infinite
number of open homotopy classes of Gauss words.
This follows from Proposition~\ref{prop:infinite} and the fact that if
two Gauss words are not homotopic, they cannot be open homotopic. 
\section{Virtual knots}\label{sec:virtualknots}
We have shown that homotopy of Gauss words is non-trivial.
In this section we interpret this fact in terms of virtual knot
diagrams. 
We start by briefly recalling some definitions.
\par
A virtual knot diagram is an immersion of an oriented circle
in a plane with a finite number of self-intersections.
These self-intersections are limited to transverse double points.
We call them crossings.
There are three types of crossings which we draw differently in order to
distinguish them.
These crossing types are shown in Figure~\ref{fig:crossings}.
\begin{figure}[hbt]
\begin{center}
\includegraphics{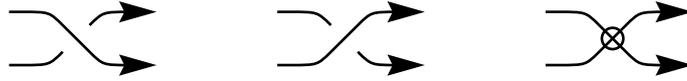}
\caption{The three types of crossing in a virtual knot
 diagram: real positive (left), real negative (middle) and virtual (right)}
\label{fig:crossings}
\end{center}
\end{figure}
\par
Virtual knots can be defined as the equivalence classes of
virtual knot diagram under a set of diagrammatic moves.
These moves include the Reidemeister moves of classical knot theory and
some other similar moves in which some crossings are virtual.
Definitions of all these moves are given, for example, in
\cite{Kauffman:VirtualKnotTheory}.
In this paper we collectively call these moves the generalized
Reidemeister moves.
\par
We define two moves on a single real crossing in a virtual
knot diagram. 
The first is called the crossing change.
It allows us to replace a positive real crossing with a negative one or
vice-versa.
This move is shown in Figure~\ref{fig:cc}.
\begin{figure}[hbt]
\begin{center}
\includegraphics{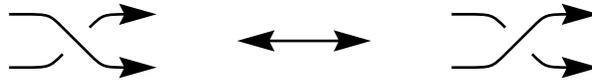}
\caption{The crossing change}
\label{fig:cc}
\end{center}
\end{figure}
\par
In classical knot theory this move gives an unknotting operation.
That is, any classical knot diagram can be reduced to a diagram with no
crossings by a sequence of Reidemeister moves and crossing changes.
On the other hand, this move is not an unknotting operation for virtual
knots. 
This is because considering virtual knots modulo this move is
equivalent to considering virtual strings and we know that non-trivial
virtual strings exist (see for example \cite{Turaev:2004}).
\par
The second move we define is called a virtual switch.
It is shown in Figure~\ref{fig:vflip}.
Kauffman first defined this move in \cite{Kauffman:VirtualKnotTheory}
and he used the name virtual switch for it in
\cite{Kauffman:Detecting}.
In \cite{Kauffman:VirtualKnotTheory} he showed that the involutary
quandle, a virtual knot invariant, is invariant even under virtual
switches. 
Since there exist virtual knots with different involutary quandles, we
may conclude that the virtual switch is not an unknotting operation for
virtual knots.
\begin{figure}[hbt]
\begin{center}
\includegraphics{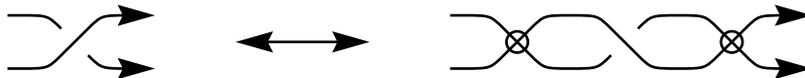}
\caption{The virtual switch}
\label{fig:vflip}
\end{center}
\end{figure}
\par
Note that both the crossing change and the virtual switch do not change
the Gauss word associated with the diagram.
The generalized Reidemeister moves do change the Gauss word associated with
the diagram.
However, the Gauss words before and after a generalized Reidemeister move
are equivalent under the moves we gave in Section~\ref{sec:gausswords}.
We have seen that there exist Gauss words that are not homotopically
trivial. 
Therefore, we can conclude that there exist virtual knots which cannot
be unknotted even if we allow the use of both the crossing change and
the virtual switch.
\par
In fact, by considering the nanoword representation of virtual knots
\cite{Turaev:KnotsAndWords}, it is easy to show that the set of homotopy
classes of Gauss words is equivalent to the set of virtual knots modulo
the crossing change and the virtual switch.
\bibliography{mrabbrev,gaussword}
\bibliographystyle{hamsplain}
\end{document}